\documentclass{amsart}
\usepackage{graphicx} 
\usepackage{ytableau}
\usepackage{amsmath,amsfonts,amssymb,amsthm,bbold,adjustbox}
\usepackage[colorlinks=true, pdfstartview=FitV, linkcolor=blue, citecolor=blue, urlcolor=darkblue]{hyperref}
\usepackage{tikz-cd}
\usepackage{tikz}
\usepackage{listings}

\DeclareMathOperator{\Ind}{Ind}

\newtheorem{lemma}{Lemma}[section]
\newtheorem{prop}[lemma]{Proposition}

\newtheorem{theorem}[lemma]{Theorem}
\newtheorem{corollary}[lemma]{Corollary}
\theoremstyle{definition}
\newtheorem{example}[lemma]{Example}
\newtheorem{remark}[lemma]{Remark}

\newcommand{\multi}[1]{\left\{\mskip-9mu\left\{#1\right\}\mskip-9mu\right\}}

\newcommand{\C}{\mathbb{C}}

\newcommand{\PP}{P\!P}

\newcommand{\ov}[1]{\overline{#1}}
\newcommand{\xt}{\tilde x}
\newcommand{\st}{\tilde s}

\newcommand{\lcr}{\{\!\!\{}
\newcommand{\rcr}{\}\!\!\}}

\newdimen\squaresize \squaresize=16pt
\newdimen\thickness \thickness=0.4pt

\def\square#1{\hbox{\vrule width \thickness
     \vbox to \squaresize{\hrule height \thickness\vss
        \hbox to \squaresize{\hss#1\hss}
     \vss\hrule height\thickness}
\unskip\vrule width \thickness}
\kern-\thickness}

\def\vsquare#1{\vbox{\square{$#1$}}\kern-\thickness}

\def\young#1{
\vbox{\smallskip\offinterlineskip
\halign{&\vsquare{##}\cr #1}}}

\def\thisbox#1{\kern-.09ex\fbox{#1}}
\def\downbox#1{\lower1.200em\hbox{#1}}

\usepackage[colorinlistoftodos]{todonotes}

\title[Rook characters as symmetric functions]{Rook characters as symmetric functions}
\author[R. Orellana, A. Wilson, M. Zabrocki]{Rosa Orellana, Alexander N. Wilson, Mike Zabrocki}
\date{\today}
\begin{document}

\begin{abstract}
We give several characterizations of an inhomogeneous basis of the ring of symmetric functions
whose evaluations are the character values of the irreducible representations
of the rook monoid (symmetric inverse semigroup).
Using Schur--Weyl duality, we show that the transition coefficients of this basis with the power symmetric basis are the characters of the propagating
partition algebra (dual symmetric inverse monoid algebra). In addition, the structure coefficients of this basis are equal to the coefficients in the smash (or Heisenberg) product of Schur functions.
\end{abstract}

\maketitle

\section{Introduction}
In \cite{OZ_character}, the authors constructed an inhomogeneous basis of symmetric functions whose evaluations are symmetric group characters. This perspective extends and complements the classical Frobenius characteristic map by introducing new bases that capture characters more directly, allowing characters to be handled with the tools of symmetric function theory. In addition, the new basis has a natural connection to the Kronecker and restriction problems. 

The rook monoid (also known as the symmetric inverse semigroup), $R_n$,
was introduced by Wagner \cite{Wagner52}
as an extension of the symmetric group that models partial permutations.
Its representation theory was studied by Munn \cite{M1957}
and Solomon \cite{solomon2002rep} and is closely related to that of the symmetric group.
In this paper we study properties of an inhomogeneous
basis of the ring of symmetric functions,
$\{ \xt_\lambda \}$, that evaluates to the irreducible characters of the rook monoid.
For this reason we call this basis the \emph{rook character basis}.
It first appeared implicitly in the
work of \cite{AS_specht} as characters of elements $[ M_\mu^t ]$ in the representation ring
${\rm Rep}(S_t)$. Later they appeared explicitly
as a basis of the ring of symmetric functions $\{\xt_\lambda\}$ defined by a power sum expansion
\cite[Equation (4)]{OZ_hopf}.  For additional
information on character bases, see the references
\cite{mitchell_lee, OZ_products, OZ_character, thibon}.

In this paper, we are interested in developing the properties of the $\xt$-basis
as we believe this basis might provide an intermediary step in tackling the
restriction and Kronecker problems.  The main result of this paper is the
following theorem, which summarizes formulae derived in this paper that
can be used to calculate the elements $\xt_\lambda$.
Condition \eqref{item:def_xt}
is the definition of the basis presented in Section \ref{sec:character bases}.
\begin{theorem} \label{theorem:TFAE}
Fix an integer $n$.  For partitions $\lambda, \mu$ partitions of $n$,
let $\chi^\lambda_{S_n}(\mu)$ be the irreducible symmetric group character and
for $\gamma, \nu$ partitions of integers smaller than or equal to $n$.  Let
$\chi^\gamma_{R_n}(\nu)$ be the irreducible rook character each indexed by the partition $\gamma$
and evaluated at an element of cycle type $\nu$.
The rook character basis is defined by any of the following equivalent conditions.
\begin{enumerate}
\item \label{item:def_xt} For each positive integer $n$ and $\lambda \vdash n$,
\[
\xt_\lambda := \sum_{\gamma \vdash n} \chi^\lambda_{S_n}(\gamma) \frac{{\overline{\mathbf p}}_\gamma}{z_\gamma}~.
\]
where ${\overline{\mathbf p}}_\gamma$ is defined in Equation \eqref{eq:def_power}.
\item \label{item:Sn_characters} For any fixed partition $\lambda$,
$\xt_\lambda$ is maximum degree $n$ and for all partitions $|\mu|<|\lambda|$,
$\xt_\lambda[\Xi_\mu] = 0$ and if $|\mu|=|\lambda|$, then $\xt_\lambda[\Xi_\mu] = \chi_{S_n}^\lambda(\mu)$.
See the introduction of Section \ref{sec:character bases} for the definition of $\Xi_\mu$.
\item \label{item:rook_characters} For any fixed partition $\lambda$ and all partitions $|\mu|\geq|\lambda|$,
$\xt_\lambda[\Xi_\mu]= \chi_{R_n}^\lambda(\mu)$, where $\chi_{R_n}^\lambda$ are the irreducible characters of the rook monoid.
\item \label{item:dual_basis} The set $\{ \xt_\lambda \}$ is the graded dual basis to
$\{ s_\lambda[\Omega - 1] \}$ with respect to the scalar product on symmetric functions.
\item \label{item:structure} The set $\{ \xt_\lambda \}$ is the unique basis whose
structure coefficients are the same as the
smash (or Heisenberg) product coefficients on Schur functions
and $\xt_{1^r} = e_r$ for all $r \geq 1$.
\item \label{item:ms_tableaux} The set $\{ \xt_\lambda \}$ is the unique basis such that
\[
h_\mu = \sum_\lambda \mathcal{K}_{\lambda\mu} \xt_\lambda
\]
where $\mathcal{K}_{\lambda\mu}$ is the number of multiset tableaux of shape $\lambda$ and content $\mu$.
\end{enumerate}
\end{theorem}
The characterizations in \eqref{item:Sn_characters} and \eqref{item:rook_characters} follow
from Lemma \ref{lemma:evaluation}, Proposition \ref{prop:evaluation_of_xt} and Equation \eqref{eq:evaluating_xt}.
Additional information on the characterizations in
\eqref{item:dual_basis}, \eqref{item:structure} and \eqref{item:ms_tableaux}
can be found in Section \ref{sec:pleth_notation} and \ref{section:multiset} where they are restated as
Theorems \ref{thm:dual_basis} and Proposition \ref{prop:dual_basis},
Theorem \ref{theorem:structure_characterization}
and Theorem \ref{thm:combinatorial} respectively.

The basis $\{\xt_\lambda\}$ has connections with recent work on the representation theory
of the propagating partition algebra (dual symmetric inverse monoid algebra)
\cite{EEF08, FL98, M07, MS21, MS25},
the rook monoid \cite{M1957, solomon2002rep, Wagner52, xiao2016},
and the Schur--Weyl duality of these two algebraic structures \cite{grood2002, grood2006, KM2008}.
Because of the Schur--Weyl duality between the rook algebra and the propagating partition algebra,
we have that the transition coefficients from the power sum basis to the rook character
basis are the characters of the propagating partition algebra (Theorem \ref{thm:power_sum_character}). As a corollary of this result, we are able to compute the characters of the propagating partition algebra using symmetric functions; in addition, we provide plethystic formulae for two factorizations of the character table of the propagating partition monoid (Proposition \ref{prop:plethystic_formula_for_factorization}).

The characters of the rook algebra as elements of the ring of symmetric functions
are interesting because the structure coefficients
are closely related to the reduced Kronecker coefficients, but their combinatorics seems slightly simpler.
The structure coefficients of the $\xt$-basis interpolate between
the Kronecker and Littlewood-Richardson
coefficients.
Both the algebra and the
combinatorics do not seem to have a dependence on the vector space in the Schur--Weyl
duality being sufficiently large.

Section \ref{sec:notation} of this paper introduces notation that will be used in
the remaining sections. Section \ref{sec:character bases} defines the rook character basis
via an expansion in the power sum basis of symmetric functions.
We then review the basic properties and characters of the rook monoid in Section \ref{sec:rook monoid} and of the propagating partition monoid in Section \ref{sec:prop_partitions}.
Along the way we make connections to symmetric functions and note how they can be used as a tool
to compute character values.  Finally, we provide some calculations in symmetric
functions using plethystic notation in Section \ref{sec:pleth_notation}.
In section \ref{section:multiset}, we show that a multiset generalization of the Kostka numbers arise as transition coefficients between the homogeneous and rook character bases.

\subsection{Acknowledgements} 
This material is based upon work supported by the National Science Foundation under Grant No. DMS-1929284 while the authors were in residence at the Institute for Computational and Experimental Research in Mathematics in Providence, RI, during the “Categorification and Computation in Algebraic Combinatorics” semester program in Fall 2025.

RO was supported by NSF grant DMS-2452044, AW and MZ were supported by NSERC/CRSNG.

\section{Notation}
\label{sec:notation}

\subsection{Partitions}\label{subsec:partitions}
A partition of the integer $k$ is a sequence of positive integers $\lambda = (\lambda_1, \lambda_2, \ldots, \lambda_\ell)$
such that $k = \lambda_1 + \lambda_2+\cdots+ \lambda_\ell$
and $\lambda_1 \geq \lambda_2 \geq \cdots \geq \lambda_\ell > 0$.
The set of cells of a partition is denoted
${\mathrm{cells}}(\lambda) := \{ (i,j) : 1 \leq i \leq \lambda_j, 1 \leq j \leq \ell(\lambda) \}$.
We write $\lambda\vdash k$ to denote that $\lambda$ is a partition of $k$.
The size of the partition is denoted $|\lambda| = k$
and the length is $\ell(\lambda) := \ell$.  The $\lambda_i$ are referred to as the parts of the
partition and we will sometimes use the notation $\lambda = (1^{m_1}2^{m_2}\cdots k^{m_k})$
to indicate that for each $1 \leq i \leq k$, the parts of size $i$ occur $m_i$ times and
denote this multiplicity by $m_i(\lambda)$.
If $i>\ell(\lambda)$, we take the convention that $\lambda_i=0$.

A skew partition is a pair of partitions, denoted $\lambda/\mu$, such that $\lambda_i \geq \mu_i$ for
each $1 \leq i \leq \ell(\mu)$.
We will use the notation $\lambda/\mu \in \mathcal{H}$
(read ``$\mu$ differs from $\lambda$ by a horizontal
strip'') to indicate that $\lambda_{i+1} \leq \mu_i$
for $1\leq i\leq \ell(\mu)$
and $\lambda\slash \mu \in \mathcal{V}$
(read ``$\mu$ differs from $\lambda$ by a verical strip'')
to indicate that $\mu_i \geq \lambda_i-1$ for $1 \leq i \leq \ell(\mu)$
and $\lambda_i=1$ for $\ell(\mu) < i \leq\ell(\lambda)$.

\subsection{Representation theory notation}

A representation of a monoid $M$ is a complex vector space $W$ upon which
$M$ acts linearly. A subspace $V$ of $W$ is a subrepresentation if for any $m\in M$
and $v\in V$, we have that $m.v\in V$. A representation $W$ is called irreducible
if it has no nontrivial subrepresentations.
We denote the irreducible representations of $M$ by $W_M^\lambda$ where $\lambda$
ranges over some indexing set (in this paper, $\lambda$ will always be a partition
with some restriction on its size).
For an irreducible representation $W_M^\lambda$, we denote by $\chi_M^\lambda$ its
character. That is, the map $\chi_M^\lambda:M\to\mathbb{C}$ where $\chi_M^\lambda(m)$ is
the trace of $m$ as a linear transformation of $W_M^\lambda$.
When $\mu$ indexes a generalized conjugacy class of $M$, we write $\chi_M^\lambda(\mu)$
for the character value on any representative of the generalized conjugacy class indexed by $\mu$.

Throughout the paper, we will use $S_n$ to denote the symmetric group, and $W_{S_n}^\lambda$ to be the irreducible representation of $S_n$ indexed by $\lambda\vdash n$.

\subsection{Symmetric functions}\label{sec:sf}

We will connect the representation theory of the rook monoid, propagating partition algebra, and the symmetric group to expressions within the ring of symmetric functions. The references \cite{Macdonald, Sagan, Stanley} provide a good introduction to symmetric functions. In what follows we summarize some of this for the convenience of the reader. 

The ring of symmetric functions $\Lambda := \mathbb{Q}[p_1, p_2, p_3, \ldots]$
are the polynomials in the power sum symmetric
function generators $p_i$ with ${\mathrm deg}(p_i) = i$ for $i\geq1$.
If the degree of every monomial in the power sum expansion of
$f$ are the same then we say that $f$ is of homogeneous degree (otherwise $f$ is
said to be of inhomogeneous degree).
When we refer to the degree of a symmetric function that may be inhomogeneous
we are referring to the maximum degree of all monomials in the power sum generators.
Two other sets of generators for this ring are the complete (homogeneous) and
elementary generators,
\[
h_r = \sum_{\lambda \vdash r} \frac{p_\lambda}{z_\lambda}
\qquad\qquad
e_r = \sum_{\lambda \vdash r} (-1)^{|\lambda|+\ell(\lambda)} \frac{p_\lambda}{z_\lambda}
\]
where if $\lambda = (1^{m_1}2^{m_2}\cdots k^{m_k})$, then
$p_\lambda=\prod_{i=1}^k{p_i}^{m_i}$ and $z_\lambda = \prod_{i=1}^k m_i! i^{m_i}$.  We will use the convention that
$p_0 =e_0 = h_0= 1$ and $p_{-i} = e_{-i} = h_{-i} = 0$ for $i>0$.
We will also use the Cauchy element
\[
\Omega = 1 + \sum_{n \geq 1} h_n
\]
as an element in the completion of the symmetric functions.

Four standard bases of the symmetric functions
are the $\{ s_\lambda \}$ Schur, $\{ h_\lambda \}$ homogeneous, $\{ e_\lambda \}$ elementary
and $\{ p_\lambda \}$ power sum bases.  Each of these bases are of homogeneous degree,
however in the next section we will also give some additional
background on bases of inhomogeneous degree \cite{OZ_hopf, OZ_character}.

The Hall inner product will be used to extract coefficients in a symmetric function expression. Recall that for this inner product the Schur functions are self dual and the power sum basis is self dual up to a scalar multiple:
\[
\left< s_\lambda, s_\mu \right> = \frac{1}{z_\lambda} \left< p_\lambda, p_\mu \right> = \delta_{\lambda\mu}~.
\]


To properly state the results in this paper we will use the operation
of plethysm.  For an element $f \in \Lambda$, we define $p_k[f]$ to be $f$ with $p_r$ replaced by $p_{kr}$.
We then define $p_\lambda[f] = p_{\lambda_1}[f] p_{\lambda_2}[f] \cdots p_{\lambda_\ell}[f]$
and if $g = \sum_{\gamma} a_\gamma p_\gamma$ for coefficients $a_\gamma \in \mathbb{Q}$,
then $g[f] = \sum_{\gamma} a_\gamma p_\gamma[f]$.

\section{Inhomogeneous Character Bases}
\label{sec:character bases}

Given a multiset $\Xi=\multi{z_1,z_2,\ldots,z_\ell}$ of complex numbers and a
symmetric function $f\in\Lambda$, let $f[\Xi]\in\C$ be the evaluation of
$f$ at this multiset where the evaluation is calculated by
replacing each $p_i$ in $f$ with the sum of $\xi^i$, for every $\xi\in \Xi$, since we have that $p_i[\Xi] = \sum_{k=1}^\ell {z_k}^i$.

Let $\Xi_\mu$ be the multiset of eigenvalues of a permutation matrix representing a
permutation with cycle type $\mu$.
In this section, we introduce some inhomogeneous bases of the space of symmetric
functions that have particularly nice evaluations at $\Xi_\mu$.
The evaluations at these eigenvalues will characterize the bases in the
sense of the following lemma.

\begin{lemma}\label{lemma:evaluation}
(see \cite[Proposition 39, Corollary 41]{OZ_character})
Let $f,g \in \Lambda$ be symmetric functions
of degree less than or equal to some positive integer $n$.
Assume that $$f[\Xi_\gamma] = g[\Xi_\gamma]$$ for all partitions
$\gamma$ such that $|\gamma| \leq n$ (or alternatively for all $|\gamma| > m$ for some fixed
integer $m$ greater than the degree of $f$ and $g$), then
$$f=g$$
as elements of $\Lambda$.
\end{lemma}

Define an inhomogeneous basis of the symmetric functions as
\begin{equation}\label{eq:def_power}
{\overline{\mathbf p}}_{\lambda} := \prod_{i\ge 1} i^{m_i} \prod_{r = 0}^{m_i-1} \left( \Big( \frac{1}{i} \sum_{d|i} {\boldsymbol \mu}(i/d) p_d \Big) - r \right)
\end{equation}
where $\lambda=(1^{m_1}2^{m_2}\cdots \ell^{m_\ell})$ and ${\boldsymbol \mu}$ is the M\"obius function on integers.

\begin{lemma}\label{lem:evaluation_of_pbar}
	For any partitions $\lambda$ and $\alpha$,
	\[\overline{\mathbf p}_\lambda[\Xi_\alpha]=\begin{cases}
		0 & \text{ if $m_i(\alpha)<m_i(\lambda)$ for any $i$}\\
		\frac{z_\alpha}{z_{\alpha-\lambda}} & \text{ otherwise}
	\end{cases}\]
	where $\alpha-\lambda$ is the partition in which the part $i$ appears
	with multiplicity $m_i(\alpha)-m_i(\lambda)$.
\end{lemma}

\begin{proof}
	First note that by \cite[Lemma 5.10.1]{lascoux},
	\[p_d[\Xi_{(r)}]=\begin{cases}r & \hbox{ if }r|d\\0&\hbox{ otherwise}\end{cases},\]
	so $p_d[\Xi_\alpha]=\sum_{d'|d}d'm_{d'}(\alpha)$.
	By M\"obius inversion,
	\[\frac{1}{i}\sum_{d|i}{\boldsymbol \mu}(i/d)\left(\sum_{d'|d}d'm_{d'}(\alpha)\right)=m_i(\alpha).\]
	Therefore evaluating Equation~\eqref{eq:def_power} at $\Xi_\alpha$, we obtain
	\begin{align*}
		\overline{\mathbf p}_\lambda[\Xi_\alpha]&=\prod_{i\ge 1} i^{m_i(\lambda)} \left(m_i(\alpha)\right)_{m_i(\lambda)}~,
	\end{align*}
    where the notation $(x)_k := x(x-1)\cdots (x-k+1)$.
	Note that if $m_i(\alpha)<m_i(\lambda)$ for any
	$i$, then $\left(m_i(\alpha)\right)_{m_i(\lambda)}=0$, hence $\overline{\mathbf p}_\lambda[\Xi_\alpha]=0$.
	In the case that $m_i(\alpha)\geq m_i(\lambda)$ for all $i$ then

	\begin{align*}
		\overline{\mathbf p}_\lambda[\Xi_\alpha]&=\prod_{i\ge 1} i^{m_i(\lambda)} \left(m_i(\alpha)\right)_{m_i(\lambda)}\\
												&=\prod_{i\ge 1} i^{m_i(\lambda)-m_i(\alpha)}i^{m_i(\alpha)} \frac{m_i(\alpha)!}{(m_i(\alpha)-m_i(\lambda))!}
												=\frac{z_\alpha}{z_{\alpha-\lambda}}.\qedhere
	\end{align*}
\end{proof}

\begin{prop}\label{prop:evaluation_of_pbar}
	For any partitions $\lambda$ and $\alpha$,
	\[\overline{\mathbf p}_\lambda[\Xi_\alpha]=\langle p_\lambda h_{|\alpha|-|\lambda|},p_\alpha\rangle.\]
\end{prop}

\begin{proof}
    Since $\left< p_\mu, p_\alpha \right>=z_\alpha$ if $\mu=\alpha$ and $0$ otherwise, we have that
    $\left< p_\lambda h_{|\alpha|-|\lambda|}, p_\alpha\right>=0$ unless $m_i(\alpha)\geq m_i(\lambda)$
    for each $i$.
    In this case,
\[p_\lambda h_{|\alpha|-|\lambda|} = \sum_{\gamma \vdash |\alpha|-|\lambda|} \frac{p_\lambda p_\gamma}{z_\gamma}\]
    and $\langle p_\lambda h_{|\alpha|-|\lambda|},p_\alpha\rangle = \frac{z_\alpha}{z_{\alpha-\lambda}}$.
    The result follows from Lemma \ref{lem:evaluation_of_pbar}.
\end{proof}

Now that we have related the classical basis $\{p_\lambda\}$ to this inhomogeneous
basis $\{\overline{\mathbf p}_\lambda\}$, we use it
to describe the inhomogeneous basis that analogously relates to the Schur functions $\{s_\lambda\}$.
This basis appeared implicitly in the work of Assaf and Speyer \cite{AS_specht} when they used the module
\[M_\lambda^t=\Ind_{S_{|\lambda|}\times S_{t-|\lambda|}}^{S_t} W^\lambda_{S_{|\lambda|}} \times W^{(t-|\lambda|)}_{S_{t-|\lambda|}}\]
as a basis of the representation ring of the symmetric group $S_t$.
In \cite{OZ_hopf}, the first and third authors of this paper defined the symmetric function elements that correspond to their characters through the following power sum expansion.
For an integer $n$ and a partition $\lambda \vdash n$,
\begin{equation}\label{eq:xt_to_power}
\xt_\lambda := \sum_{\gamma \vdash n} \chi^\lambda_{S_n}(\gamma) \frac{{\overline{\mathbf p}}_\gamma}{z_\gamma}~.
\end{equation}

\begin{prop}\label{prop:evaluation_of_xt}
	For any partitions $\lambda$ and $\alpha$,
	\[\xt_\lambda[\Xi_\alpha]=\langle s_\lambda h_{|\alpha|-|\lambda|},p_\alpha\rangle.\]
\end{prop}

\begin{proof}
	By Equation~\eqref{eq:xt_to_power} and  Proposition~\ref{prop:evaluation_of_pbar},
	\begin{align*}
		\xt_\lambda[\Xi_\alpha]&=\sum_{\gamma \vdash n} \chi^\lambda_{S_n}(\gamma) \frac{{\overline{\mathbf p}}_\gamma[\Xi_\alpha]}{z_\gamma}\\
		&=\sum_{\gamma \vdash n} \chi^\lambda_{S_n}(\gamma) \frac{\langle p_\gamma h_{|\alpha|-|\lambda|},p_\alpha\rangle}{z_\gamma}
		=\langle s_\lambda h_{|\alpha|-|\lambda|},p_\alpha\rangle. \qedhere
	\end{align*}
\end{proof}


For a fixed $n$, we may define the Frobenius image of a symmetric function by
\[
\phi_n( f) := \sum_{\mu \vdash n} f[\Xi_\mu] \frac{p_\mu}{z_\mu}.
\]
A result of Littlewood \cite{Littlewood} and Scharf and Thibon \cite[Theorem 5.1]{ST} states that for a symmetric function
$g$ of degree $n$, and symmetric function $f$, that the operator $\phi_n$ is dual to
plethysm in the following sense.  Say that $f \in \Lambda$ is of bounded degree and $g \in \Lambda$ of
homogeneous degree $n$, then:
\begin{equation}\label{eq:scharfthibon}
\left< \phi_n(f), g \right> = \left< f, g[\Omega ]\right>~.
\end{equation}
If $g \in \Lambda$ with $g$ not homogeneous of degree zero, then
$g[\Omega] = g[1 + h_1 + h_2 + \cdots]$ is not an element of $\Lambda$. Instead, like
$\Omega$, it lies in the graded completion of $\Lambda$
because it has unbounded degree (see Remark 1 of Section 2 of \cite{Macdonald}).
However, since the symmetric functions are a graded space and $f$ is specified to be
of bounded degree,
computations such as Equation \eqref{eq:scharfthibon} can be completed
up to a finite fixed degree and computations of elements in the completion can be truncated.

By Proposition \ref{prop:evaluation_of_xt},
for all partitions $\lambda$ and for any $n\geq0$,
\begin{equation}\label{eq:frob_image_xt}
\phi_n(\xt_\lambda) = h_{n-|\lambda|} s_\lambda
\end{equation}
where we recall that $h_i=0$ for $i<0$.

As noted in \cite[Remark 1]{OZ_hopf}, the $\{ \xt_\lambda \}$ are the
irreducible characters of the
rook algebra in the same way that the irreducible characters of the symmetric group
are defined as a basis in \cite{OZ_character}.
Additional details connecting the character values of the $\xt_\lambda$ to results
in the literature are in Section \ref{sec:rook monoid}
in order to motivate further study of the $\{ \xt_\lambda \}$ as
a basis of symmetric functions.


\section{The rook monoid characters and structure coefficients}
    \label{sec:rook monoid}

A partial permutation on $[n]$ is a pair $(D,f)$
of a subset $D\subseteq[n]$ called the domain and
an injective function $f:D\to [n]$.
The monoid of partial permutations on $[n]$
is called the symmetric inverse monoid.
This monoid is isomorphic to the monoid of $n\times n$ matrices
with at most one $1$ in each row and column with remaining entries all
zero with the usual matrix multiplication.
Because these matrices resemble the placement of non-attacking
rooks on a chessboard, this monoid is often called the {\it rook monoid}, $R_n$.
Going forward, we will use the name rook monoid and think of the elements
interchangeably as partial permutations or as these rook matrices.

Green's relations are five equivalence relations that characterize the elements of a semigroup by
the principal ideals they generate. For our purposes, we are interested only in the $\mathcal{J}$-classes
of a monoid $M$:
Two elements $m,n\in M$ are in the same $\mathcal{J}$-class if and only if $MmM=MnM$.
A $\mathcal{J}$-class is called regular if it contains an idempotent.

Given an idempotent $e\in M$, the maximal subgroup of $e$, denoted $G_e$,
is the largest subset of $M$ containing $e$ for which the monoid multiplication
restricts to a group structure with $e$ as the identity.

There is a notion of generalized conjugacy classes for monoids (see \cite[Section 7.1]{steinberg} for more details).
For our purposes, we only need to know that generalized conjugacy classes of an inverse monoid
$M$ are in correspondence with the conjugacy classes of its maximal subgroups $G_e$ where $e$
ranges over representative idempotents for the regular $\mathcal{J}$-classes of $M$ \cite[Proposition 7.4]{steinberg}.

The generalized conjugacy classes of $R_n$ are indexed by partitions $\lambda$ with $0\leq|\lambda|\leq n$.
A representative of the conjugacy class indexed by $\lambda\vdash k$ is given by any element of $S_k$ with
cycle type $\lambda$ written as a permutation matrix padded by enough zeros to create an $n\times n$ matrix.
We say that $\sigma\in R_n$ has cycle type $\lambda$
if it belongs to the generalized conjugacy class indexed by $\lambda$
(and this is sometimes referred to as the Munn class \cite{M1957}).
Note that the nonzero eigenvalues of $\sigma$ as a matrix are the multiset $\Xi_\lambda$.
For more details about generalized conjugacy in the rook monoid see \cite[Section 3.2]{steinberg}.

The generalized conjugacy classes of a finite monoid $M$ are
in correspondence with irreducible representations of $M$ \cite[Theorem 7.10]{steinberg},
so we can construct a square character table for a finite monoid
in much the same way that we construct one for a finite group. Each column is indexed by a generalized
conjugacy class of $M$, and each row is indexed by an irreducible representation of $M$.
In \cite{M1957}, W.~D.~Munn studied this character table for $R_n$.
In \cite{solomon2002rep}, L.~Solomon gave the following formula
for computing these characters in terms of those for the symmetric group.

\begin{lemma}\cite[Proposition 3.11]{solomon2002rep}\label{lem:Rn_char_computation}
	Fix $n$ and let $\lambda,\alpha$ be partitions with $|\lambda|,|\alpha|\leq n$. Then
	\[\chi^{\lambda}_{R_n}(\alpha)=\sum_{\mu/\lambda\in\mathcal{H}}\chi_{S_{|\alpha|}}^\mu(\alpha)\]
	where the sum is over partitions $\mu\vdash |\alpha|$ such that $\mu/\lambda$ is a horizontal strip.
\end{lemma}

Since $\chi_{S_{|\alpha|}}^\mu(\alpha) = \left< s_\mu, p_\alpha \right>$
and because $h_{|\alpha|-|\lambda|} s_\lambda = \sum_{\mu/\lambda\in\mathcal{H}} s_\mu$,
in symmetric function notation this lemma is equivalent to
\begin{equation}\label{eq:evaluating_xt}
\chi^{\lambda}_{R_n}(\alpha)
= \left< h_{|\alpha|-|\lambda|} s_\lambda, p_\alpha \right>
= \xt_\lambda[\Xi_\alpha]~.
\end{equation}

Note that the character value $\chi^\lambda_{R_n}(\alpha)$ is
independent of $n$ so long as $n\geq\max\{|\lambda|,|\alpha|\}$.

\begin{remark}
Equations \eqref{eq:xt_to_power}, \eqref{eq:evaluating_xt}
and Lemma \ref{lemma:evaluation} together
imply that the first three characterizations
of Theorem \ref{theorem:TFAE} are all equivalent.
The $\xt_\lambda$ are the only symmetric functions of
degree $|\lambda|$ that evaluate to rook characters
and this characterizes the elements.
\end{remark}

\begin{example}\label{ex:xt_to_power}
Because the character of a representation evaluated at the identity element of the monoid gives
the dimension of the representation, the dimension of the irreducible representation of $R_n$
indexed by $\lambda$ can be computed by evaluating $\xt_\lambda$ at $\Xi_{(1^n)}$. That is,
\[\dim(W_{R_n}^\lambda)=\xt_\lambda(1^n).\]

As an explicit example computed from formula \eqref{eq:xt_to_power},
\[\xt_{21}=p_{1}-p_{11}+\frac13p_{111}-\frac13p_{3},\]
so the dimension of $W_{R_4}^{21}$ is \[\xt_{21}(1,1,1,1)=4-16+\frac{1}{3}\cdot64-\frac{1}{3}\cdot 4=8.\]
\end{example}

\begin{corollary}\label{corollary:xtr_equals_er}
For all $r \geq 1$, $\xt_{1^r} = e_r$.
\end{corollary}

\begin{proof}
The symmetric function $\xt_{1^r}$ is an element
such that if $|\alpha|<r$, then $\xt_{1^r}[\Xi_\alpha]=0$ and
if $|\alpha|=r$, then $\xt_{1^r}[\Xi_\alpha] = \chi^{(1^r)}_{R_r}(\alpha)
= \chi^{(1^r)}_{S_r}(\alpha)$ which is equal to the sign of
a permutation of cycle type $\alpha$.

Recall \cite[p. 19]{Macdonald}
that $e_r[X] = \sum_{i_1 < i_2 < \cdots < i_r} x_{i_1} x_{i_2} \cdots x_{i_r}$
and that if the alphabet $X$ has cardinality less than $r$, then $e_r[X]=0$.
We know therefore that $e_r[\Xi_\alpha]=0$ if $|\alpha|<r$ and if $|\alpha|=r$ then
$\Xi_\alpha$ has exactly $r$ non-zero eigenvalues and
$e_r[\Xi_\alpha]$ is equal to the product of the eigenvalues in $\Xi_\alpha$.
That is $e_r[\Xi_\alpha]$ is the product of the eigenvalues
of a permutation matrix of cycle type $\alpha$ (and this in turn is equal to the
determinant of the permutation matrix and therefore equal to the sign of
the permutation).  Therefore, by Lemma \ref{lemma:evaluation}, $\xt_{1^r}=e_r$.
\end{proof}

Given two irreducible characters $\chi^\lambda_{R_n}$ and $\chi^\mu_{R_n}$ of $R_n$, their pointwise product $\chi^\lambda_{R_n}\cdot\chi^\mu_{R_n}$ is also a character of $R_n$ corresponding to the character of
a tensor product of irreducible rook monoid modules with a diagonal action.
This character then decomposes into a nonnegative sum of irreducible characters.
We write $k_{\lambda\mu}^\gamma$ for the multiplicity of $\chi^\gamma_{R_n}$ in this decomposition.
That is,
\begin{equation}\label{eq:character_structure_coefficients}
\chi^\lambda_{R_n}\cdot\chi^\mu_{R_n}=\sum_{\gamma:|\gamma|\leq n}k_{\lambda\mu}^\gamma\chi^\gamma_{R_n}.
\end{equation}

The following proposition is a consequence of
Lemma \ref{lemma:evaluation}.  It states that the $\{\xt_\lambda\}$
basis has the same structure coefficients as this internal product of the
irreducible rook characters.

\begin{prop}\label{prop:structure}
Let $\lambda, \mu$ be partitions,
\[
\xt_\lambda \xt_\mu = \sum_{\gamma:|\gamma|\leq |\lambda|+|\mu|}k_{\lambda\mu}^\gamma \xt_\gamma~.
\]
\end{prop}

\begin{proof}
Let $n\geq |\lambda|+|\mu|$.
The $k_{\alpha\beta}^\gamma$ are the structure coefficients of the characters
of the rook monoid, hence by Equation \eqref{eq:character_structure_coefficients},
for any partition $\tau$ such that $|\tau| \leq n$,
\[
\chi^{\lambda}_{R_n}(\tau) \chi^{\mu}_{R_n}(\tau)
= \sum_{\gamma} k_{\lambda\mu}^\gamma \chi^\gamma_{R_n}(\tau).
\]
We have that for partitions $\lambda$, $\tilde{x}_\lambda[\Xi_\nu] = \chi^{\lambda}_{R_n}(\nu)$,
then by Lemma \ref{lemma:evaluation} it follows that the $\tilde{x}$-basis have the same structure coefficients.
\end{proof}

In \cite{AFM} the authors introduced what they referred to as the
Heisenberg product (smash product) on species and developed notation for its
computation on the Hopf algebras of permutations, quasisymmetric functions,
noncommutative symmetric functions and symmetric functions.
Their motivation for introducing this product was to both interpolate and unify
the classical external and internal products on Hopf algebras.

Follow up work on this product by Ying \cite{Ying20, Ying22}
gave more explicit formulae for the smash product on Schur
functions in terms of Littlewood-Richardson coefficients
and Kronecker coefficients.
In particular, the smash product coefficients interpolate between
the Littlewood-Richardson coefficients (the highest degree components) and the Kronecker
coefficients (the lowest degree components).

In a surprise connection, we noticed that Equation \eqref{eq:rook_kronecker}
for the structure coefficients of pointwise product of
the rook characters also appears as the coefficients of the {\it smash product}
of Schur functions \cite{Ying22} as Lemma 4.2.
We formalize this statement as the following proposition, the proof of this result is a comparison of \cite[Theorem 1]{MS25} and \cite[Lemma 4.2]{Ying22}:

\begin{prop}\label{prop:xtproduct_equals_smashproduct}
For partitions $\alpha, \beta$ and $\gamma$ such that $|\gamma| \leq |\alpha|+|\beta|$,
the coefficient of $\xt_\gamma$ in the product $\xt_\alpha \xt_\beta$
is equal to the coefficient of $s_\gamma$ in the product $s_\alpha \# s_\beta$
(following Definition 2.1 of \cite{Ying22}).
\end{prop}

\section{The propagating partition monoid and characters}
\label{sec:prop_partitions}

The rook monoid $R_n$ acts naturally on $\C^n$ by multiplication by rook matrices; hence $R_n$ act diagonally on
the tensor power $(\C^n)^{\otimes k}$. By \cite{KM2008}, this action of $R_n$ is Schur--Weyl
dual to the action of the propagating partition monoid, denoted $\PP_k$, on the same space so long as $n\geq k$.
Using the decomposition of $(\C^n)^{\otimes k}$ as an $R_n$-module in \cite[Example 3.18]{solomon2002rep} and the double-centralizer
theorem, we have the decomposition \begin{equation}\label{eq:SW_decomp}
	(\C^n)^{\otimes r}=\bigoplus_{\lambda:0\leq|\lambda|\leq\max\{n,k\}} W_{R_n}^\lambda\otimes W_{\PP_k}^\lambda
\end{equation} as an $R_n\times \PP_k$-module.

Explicitly, $\PP_k$ consists of set partitions $\pi$ of $\{1,2,\ldots,k\}\cup\{\overline1,\overline2,\ldots,\overline{k}\}$
such that every block of $\pi$ contains at least  one of each unbarred and barred number.
Such a set partition can be represented graphically
by a diagram consisting of a top row of dots labeled $\{1,2,\ldots,k\}$ and a bottom row of dots labeled
$\{\overline1,\overline2,\ldots,\overline{k}\}$ whose connected components give the blocks of $\pi$.

\begin{example} \label{ex:diagram_PP5}
	The set partition $\pi=\{\{1,\overline1,\overline2\},\{2,4,\overline3\},\{3,\overline4\},\{5,\overline5\}\}\in\PP_5$ is given diagrammatically by 
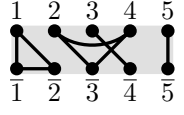
\[
	\begin{tikzpicture}[xscale=.5,yscale=.5,line width=1.25pt] 
                \foreach \i in {1,...,5}  { \path (\i,1.25) coordinate (T\i); \path(\i,1.8) node {$\i$}; \path (\i,.25) coordinate (B\i); \path(\i,-.3) node {$\ov{\i}$}; } 
                \filldraw[fill= black!12,draw=black!12,line width=4pt]  (T1) -- (T5) -- (B5) -- (B1) -- (T1);
                \draw[black] (T1)--(B1)--(B2)--(T1);
                \draw[black] (B3)--(T2)  .. controls +(.5,-.5) and +(-.5,-.5) .. (T4)--(B3);
                \draw[black] (T3)--(B4);
                \draw[black] (T5)--(B5);
                \foreach \i in {1,...,5}  {\filldraw[fill=black,draw=black,line width = 1pt] (T\i) circle (4pt);
                			     \filldraw[fill=black,draw=black,line width = 1pt] (B\i) circle (4pt);}
       \end{tikzpicture}.\]
\end{example}

The action of $\pi\in\PP_k$ on $(\C^n)^{\otimes k}$ can be computed on a basis element
$e_{i_1}\otimes e_{i_2}\otimes\cdots\otimes e_{i_k}$ by placing the tensor factors
beneath the bottom row of the diagram of $\pi$. If any distinct vectors sit below vertices
that are connected, the result is zero. Otherwise, the vertices in the top row are labeled
with whatever vector they are connected to on the bottom row, and the resulting tensor product
on top is the output of $\pi$ acting on this basis element.

\begin{example}
	For $\pi$ the diagram in Example \ref{ex:diagram_PP5},
       we see that
       \[\pi.(e_1\otimes e_2\otimes e_3\otimes e_4\otimes e_5)=
       \begin{cases} e_1\otimes e_3\otimes e_4\otimes e_3\otimes e_5&\hbox{ if } e_1=e_2\\
       0&\hbox{ otherwise}\end{cases}.\]
\end{example}

The product of two elements of $\PP_k$ is given by their composition as maps on $(\C^n)^{\otimes k}$.
Note that the rank of $\pi$ as a map on this space is equal to $\ell(\pi)$, which leads to the useful
observation that for any $\pi_1,\pi_2\in \PP_k$, \begin{equation}\label{eq:rank_monotone}
	\ell(\pi_1\pi_2)\leq\min\{\ell(\pi_1),\ell(\pi_2)\}.
\end{equation}

\subsection{Characters of $\PP_k$}
\label{subsec:characers_PP}

By \cite[Theorem 2.2]{FL98}, the $\mathcal{J}$-classes of $\PP_k$ are $\{J_i:1\leq i\leq k\}$ where \[J_i=\{\pi\in \PP_k:\ell(\pi)=i\}.\]
For each $1\leq i\leq k$, we choose an idempotent \[e_i=\{\{1,\ov1\},\{2,\ov2\},\ldots,\{i-1,\ov{i-1}\},\{i,i+1,\ldots,k,\ov{i},\ov{i+1},\ldots,\ov{k}\}\}\in J_i.\]

In order to compute with characters of $\PP_k$, we now give an explicit set of generalized conjugacy class representatives.

Fix $k\geq 1$.
Given a permutation $\sigma\in S_i$ for $i\leq k$, define $\sigma[k]\in\PP_k$ as follows.
Consider the set partition $\{\{\sigma(1),\ov1\},\{\sigma(2),\ov2\},\ldots,\{\sigma(i),\ov{i}\}\}$.
To form $\sigma[k]$, we add $\{i+1,i+2,\ldots,k\}$ to the block containing $i$ and $\{\ov{i+1},\ov{i+2},\ldots,\ov{k}\}$ to the block containing $\ov{i}$

\begin{prop}\label{prop:maximal_subgps_of_PP}
	The map \begin{align*}
		S_i&\to G_{e_i}\\
		\sigma&\mapsto \sigma[k]
	\end{align*} is a group isomorphism between the symmetric group $S_i$ and the maximal subgroup $G_{e_i}\subseteq\PP_k$.
\end{prop}

\begin{proof}
	First note that only the identity element of $S_i$ indeed maps to the idempotent $e_i$ and the
	definition of the product in $\PP_k$ makes the map a homomorphism. We then need only show that
	the map is surjective.

	By Equation~\ref{eq:rank_monotone}, if $\pi\in G_{e_i}$, it must have at least $i$ blocks.
	Furthermore, because $\pi e_i=e_i\pi=\pi$, we must have that the values $i,i+1,\ldots,k$ are in
	the same block of $\pi$ and likewise with the values $\ov{i},\ov{i+1},\ldots,\ov{k}$.
	Taken together, these two facts imply that the values $1,2,\ldots,i-1$ must be in \emph{separate}
	blocks of $\pi$ and likewise for the values $\ov{1},\ov{2},\ldots,\ov{i-1}$.
	Hence, $\pi$ must be of the form $\sigma[k]$ for some $\sigma\in S_i$.
\end{proof}

Fix $k\geq1$ and let $\mu\vdash i\leq k$.
For $i<j$, define the cycle $c(i,j)=(i,i+1,\ldots,j)$.
Define \[\sigma_\mu=c(1,\mu_1)c(\mu_1+1,\mu_1+\mu_2)\cdots c(\mu_1+\mu_2+\cdots+\mu_{l-1}+1,i)\in S_i\] and set $d_\mu=\sigma_\mu[k]$.

\begin{theorem}
	The set \[\{d_\mu:\mu\vdash i\leq k\}\] forms a complete set of representatives for the generalized conjugacy classes of $\PP_k$.
\end{theorem}

\begin{proof}
	By \cite[Proposition 7.4]{steinberg}, there is a correspondence between generalized
	conjugacy classes of $\PP_k$ and conjugacy classes of the maximal subgroups $G_{e_i}$ for $1\leq i\leq k$.
	More precisely, each conjugacy class $[g]$ of $G_{e_i}$ is contained in a
	corresponding generalized conjugacy class of $\PP_k$.

	The elements $\{\sigma_\mu:\mu\vdash i\}$ form a set of conjugacy class representatives for each $S_i$.
	Via the correspondence in Proposition~\ref{prop:maximal_subgps_of_PP}, these representatives correspond
	to the elements $\{d_\mu:\mu\vdash i\}$.
\end{proof}

\begin{example}
	In $\PP_9$, the generalized conjugacy class representative $d_{(3,2,2)}$ is given diagrammatically by 
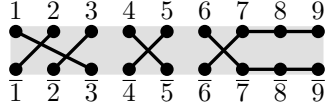
\[
	\begin{tikzpicture}[xscale=.5,yscale=.5,line width=1.25pt] 
                \foreach \i in {1,...,9}  { \path (\i,1.25) coordinate (T\i); \path(\i,1.8) node {$\i$}; \path (\i,.25) coordinate (B\i); \path(\i,-.3) node {$\ov{\i}$}; } 
                \filldraw[fill= black!12,draw=black!12,line width=4pt]  (T1) -- (T9) -- (B9) -- (B1) -- (T1);
                \draw[black] (T1)--(B3);
                \draw[black] (T2)--(B1);
                \draw[black] (T3)--(B2);
                \draw[black] (T4)--(B5);
                \draw[black] (T5)--(B4);
                \draw[black] (T6)--(B7)--(B9);
                \draw[black] (B6)--(T7)--(T9);
                \foreach \i in {1,...,9}  {\filldraw[fill=black,draw=black,line width = 1pt] (T\i) circle (4pt);
                			     \filldraw[fill=black,draw=black,line width = 1pt] (B\i) circle (4pt);}
       \end{tikzpicture}\]
\end{example}

\begin{lemma}\label{lem:bitrace_computation}
	Let $\alpha\vdash j\leq n$ and $\beta\vdash i\leq k$ and let $\sigma\in R_n$ have cycle type $\alpha$.
	The bitrace of $(\sigma,d_\beta)\in R_n\times\PP_k$ acting on $\left(\C^n\right)^{\otimes k}$ is \[p_\beta[\Xi_\alpha].\]
\end{lemma}

\begin{proof}
	Let $W\subseteq \left(\C^n\right)^{\otimes k}$ be the subspace spanned by basis elements whose last $k-i+1$ tensor factors are the same.
	That is, basis elements of the form \[e_{m_1}\otimes e_{m_2}\otimes\cdots\otimes e_{m_i}\otimes e_{m_i}\otimes \cdots \otimes e_{m_i}.\]
	Note that $d_\beta$ sends basis elements outside of $W$ to zero and the image of $d_\beta$ is contained in $W$.
	Hence, the bitrace of $(\sigma,d_\beta)$ acting on $\left(\C^n\right)^{\otimes k}$ is equal to the bitrace of $(\sigma,\ov{d}_\beta)$ on $W\cong\left(\C^n\right)^{\otimes i}$ where $\ov{d}_\beta$ is the result of restricting $d_\beta$ to entries at most $i$. Note that $\ov{d}_\beta$ is a permutation diagram for a permutation in $S_i$ with cycle type $\beta$.
	
	Schur proved that the bitrace of $(g,\gamma)\in GL_n\times S_i$ acting on $\left(\C^n\right)^{\otimes i}$ is $p_\mu(x_1,x_2,$ $\ldots,x_n)$ where $\mu$ is the cycle type of $\gamma$ and $x_1,x_2,\ldots,x_n$ are the eigenvalues of $g$.
	Specializing this formula to $(\sigma,\ov{d}_\beta)$ yields the evaluation of $p_\beta$ at the eigenvalues of $\sigma$ as an $n\times n$ matrix. That is, $p_\beta[\Xi_\alpha]$.
\end{proof}

\begin{theorem} \label{thm:power_sum_character}
Let $k$ be a positive integer and $\mu$ a partition such that $|\mu|\leq k$. Then,
\[
p_\mu = \sum_{\lambda : 1 \leq |\lambda| \leq |\mu|} \chi^\lambda_{\PP_k}(\mu) \tilde{x}_\lambda.
\]
\end{theorem}

\begin{proof}
	Let $\alpha\vdash n\geq k$.
	We can compute the bitrace of an element $(\sigma,d_\beta)\in R_n\times \PP_k$ where $\sigma$ has cycle type $\alpha$ acting on $\left(\C^n\right)^{\otimes k}$ in two ways:
	we can compute it directly using Lemma~\ref{lem:bitrace_computation} as $p_\mu[\Xi_\alpha]$, and we can compute it using the decomposition of the space as an $R_n\times \PP_k$-bimodule in Equation~\eqref{eq:SW_decomp}.
	Equating these two computations yields
	\begin{align*}
		p_\mu[\Xi_\alpha]=\sum_{\lambda:1\leq |\lambda|\leq k} \chi_{R_n}^\lambda(\alpha)\chi_{\PP_k}^\lambda(\mu).
	\end{align*}
	
	By Equation
	~\eqref{eq:evaluating_xt},
	
	\begin{align*}
		p_\mu[\Xi_\alpha]&=\sum_{\lambda:1\leq |\lambda|\leq k} \tilde{x}_\lambda[\Xi_\alpha]\chi_{\PP_k}^\lambda(\mu)\\
				   &=\left(\sum_{\lambda:1\leq |\lambda|\leq k} \chi_{\PP_k}^\lambda(\mu)\tilde{x}_\lambda\right)[\Xi_\alpha]~.
	\end{align*}
	
	Because these two symmetric functions agree evaluated at $\Xi_\alpha$ for all partitions $\alpha$ such that $|\alpha|\geq k$ and their degrees are bounded by $k$, we have by Lemma \ref{lemma:evaluation} that they agree as elements of $\Lambda$.
	
	Finally, choosing $k=|\mu|$, we see that the only $\lambda$ in the sum that contribute for any $k$ are those for which $|\lambda|\leq|\mu|$.
\end{proof}

\subsection{Decomposition Matrices}

Let $C$ be the character table of $\PP_k$ with rows and columns
indexed by partitions in graded lexicographic order i.e.
\[\emptyset, (1), (2), (1,1), (3), \ldots, (1^k).\]

\begin{example}\label{ex:PPk_char_table}
The character table for $\PP_3$ is given below.
\[\left[\begin{array}{rrrrrrr}
1 & 0 & 0 & 0 & 0 & 0 & 0 \\
0 & 1 & 1 & 1 & 1 & 1 & 1 \\
0 & 0 & 1 & 1 & 0 & 1 & 3 \\
0 & 0 & -1 & 1 & 0 & 1 & 3 \\
0 & 0 & 0 & 0 & 1 & 1 & 1 \\
0 & 0 & 0 & 0 & -1 & 0 & 2 \\
0 & 0 & 0 & 0 & 1 & -1 & 1
\end{array}\right]\]
\end{example}

Let $Y$ be the block diagonal matrix of the same dimensions
whose diagonal blocks are the character tables of $S_i$
for $0\leq i\leq k$.
The character table $C$ factors as
\[C=LY~~~~~\hbox{and}~~~~~C=YR\]
for unique upper unitriangular matrices $L$ and $R$.
Solomon \cite{solomon2002rep} exploited this fact to provide his formulas for the character
table of the rook monoid, and Steinberg \cite{steinberg_mobius} extended the method to the 
more general setting of finite inverse semigroups.
The matrix $L$ is called the decomposition matrix of $\PP_k$
because it records how an irreducible representation of
$\PP_k$ decomposes upon restriction to a maximal subgroup (see \cite[Theorem 6.5]{steinberg}).

\begin{example}
The character table in Example~\ref{ex:PPk_char_table} for $\PP_3$ factors in the following two ways
as $LY$ and $YR$ respectively.
	\begin{align*}
\left[\begin{array}{rrrrrrr}
1 & 0 & 0 & 0 & 0 & 0 & 0 \\
0 & 1 & 1 & 0 & 1 & 0 & 0 \\
0 & 0 & 1 & 0 & 1 & 1 & 0 \\
0 & 0 & 0 & 1 & 1 & 1 & 0 \\
0 & 0 & 0 & 0 & 1 & 0 & 0 \\
0 & 0 & 0 & 0 & 0 & 1 & 0 \\
0 & 0 & 0 & 0 & 0 & 0 & 1
\end{array}\right]&
\left[\begin{array}{rrrrrrr}
1 & 0 & 0 & 0 & 0 & 0 & 0 \\
 0 & 1 & 0 & 0 & 0 & 0 & 0 \\
 0 & 0 & 1 & 1 & 0 & 0 & 0 \\
0 & 0 & -1 & 1 & 0 & 0 & 0 \\
 0 & 0 & 0 & 0 & 1 & 1 & 1 \\
0 & 0 & 0 & 0 & -1 & 0 & 2 \\
0 & 0 & 0 & 0 & 1 & -1 & 1
\end{array}\right]\\
\left[\begin{array}{rrrrrrr}
1 & 0 & 0 & 0 & 0 & 0 & 0 \\
 0 & 1 & 0 & 0 & 0 & 0 & 0 \\
 0 & 0 & 1 & 1 & 0 & 0 & 0 \\
0 & 0 & -1 & 1 & 0 & 0 & 0 \\
 0 & 0 & 0 & 0 & 1 & 1 & 1 \\
0 & 0 & 0 & 0 & -1 & 0 & 2 \\
0 & 0 & 0 & 0 & 1 & -1 & 1
\end{array}\right]&
\left[\begin{array}{rrrrrrr}
1 & 0 & 0 & 0 & 0 & 0 & 0 \\
0 & 1 & 1 & 1 & 1 & 1 & 1 \\
0 & 0 & 1 & 0 & 0 & 0 & 0 \\
0 & 0 & 0 & 1 & 0 & 1 & 3 \\
0 & 0 & 0 & 0 & 1 & 0 & 0 \\
0 & 0 & 0 & 0 & 0 & 1 & 0 \\
0 & 0 & 0 & 0 & 0 & 0 & 1
\end{array}\right]
	\end{align*}
\end{example}

In the next section, we will give plethystic formulations for the entries in the matrices $L$ and $R$
as well as a plethystic description of the irreducible characters of the propagating partition monoid.

\section{Rook characters and plethystic notation}
\label{sec:pleth_notation}

In this section we prove a number of symmetric function expressions involving the rook character basis.
Since the proofs
we provide here require additional notation and identities, we introduce the required
formulae in this section.
In addition to the standard notation for symmetric functions, we will also use plethystic notation
with multiple alphabets, which allows us to encode Hopf algebra operations
on symmetric functions compactly.

For a countable alphabet $\{ x_1, x_2, x_3, \ldots \}$, we use the notation $X$ to represent the sum
$x_1 + x_2 + x_3 + \cdots$.
Then for $f \in \Lambda$, and an expression $E(x_1, x_2, x_3, \ldots)$ in these variables,
$f[E]$ represents the replacement of each
instance of $p_k$ in $f$ by $E(x_1^k, x_2^k, x_3^k, \ldots)$.
We note that the expression $f[g[X]]$
coincides with the plethysm $f[g][X]$ defined in Section~\ref{sec:sf}.
If a second set of variables $\{y_1,y_2,y_3,\ldots\}$ is needed,
we will set $Y := y_1 + y_2 + y_3 + \cdots$ and we can work with
the sum $X+Y$ and product $XY = \sum_{i,j} x_i y_j$ of alphabets.

We will express the structure coefficients of the $\xt_\lambda$ basis in terms of
the Littlewood-Richardson coefficients $c^\lambda_{\mu\nu}$ and the Kronecker coefficients $g_{\lambda\mu\nu}$,
which appear in the symmetric function expressions:
\begin{equation}\label{eq:Schur_coproducts}
c^\lambda_{\nu\mu} = \left< s_\nu[X]s_\mu[Y], s_\lambda[X+Y] \right>~~~~~\hbox{and}~~~~~
g_{\lambda\mu\nu} = \left< s_\nu[X]s_\mu[Y], s_\lambda[XY] \right>~.
\end{equation}

The dual Pieri rules are expressed by taking the value $1$ and
adding it to or (virtually) removing it from the alphabet $X$.
\begin{equation}\label{eq:pieri_duals}
s_\lambda[X + 1] = \sum_{\mu : \lambda / \mu \in \mathcal{H}} s_\mu[X]
~~~~~\hbox{and}~~~~~
s_\lambda[X - 1] = \sum_{\mu : \lambda / \mu \in \mathcal{V}} (-1)^{|\lambda/\mu|} s_\mu[X]~.
\end{equation}

The Cauchy element is the series of symmetric functions
\begin{equation}\label{eq:cauchy}
\Omega[X] = \prod_{i} \frac{1}{1-x_i} = 1 + \sum_{k\geq1} s_{(k)}[X]~.
\end{equation}
This element of the completion of the symmetric functions has the property that
\begin{equation}\label{eq:Cauchy_properties}
\Omega[X+Y] = \Omega[X]\Omega[Y]\qquad\hbox{and}\qquad\Omega[XY] = \sum_{\lambda} s_\lambda[X] s_\lambda[Y]~.
\end{equation}

The Pieri rules for Schur functions imply that
for any symmetric functions $f, g \in \Lambda$,
\begin{equation}\label{eq:Cauchy_dual}
\left< f[X+1], g[X] \right> = \left< f[X], g[X] \Omega[X] \right>~.
\end{equation}

Given these identities, we are now prepared to give the
following additional property of the $\xt$-basis.
Assume that $\{ \xt_\lambda  \}$ is the basis which
evaluates to the rook characters
(that is, satisfies Equation \eqref{eq:frob_image_xt}).

\begin{theorem} \label{thm:dual_basis}
For partitions $\lambda$ and $\mu$,
\[
\left< \xt_\lambda, s_\mu[\Omega-1] \right> = \delta_{\lambda\mu}.
\]
That is, the elements $\{ \xt_\lambda \}$ are dual to the elements
$\{ s_\mu[\Omega-1] \}$ (which are elements in the completion of the symmetric functions).
\end{theorem}

\begin{proof}
Let $\lambda \vdash n$, and $\mu \vdash m$.  We use Equations \eqref{eq:Cauchy_dual}, \eqref{eq:pieri_duals},
\eqref{eq:frob_image_xt} and \eqref{eq:scharfthibon} to compute
\begin{align*}
\left< s_\lambda, s_\mu \right> &= \left< s_\lambda[X], s_\mu[X+1-1] \right>\\
&= \left< s_\lambda[X] \Omega[X], s_\mu[X-1] \right>\\
&= \sum_{\mu/\gamma \in \mathcal{V}} (-1)^{|\mu|-|\gamma|}
\left< \phi_{|\gamma|}( \xt_\lambda ), s_{\gamma}[X] \right>\\
&= \sum_{\mu/\gamma \in \mathcal{V}} (-1)^{|\mu|-|\gamma|}
\left< \xt_\lambda, s_\gamma[\Omega] \right>\\
&= \left< \xt_\lambda, s_\mu[ \Omega - 1] \right>~.\qedhere
\end{align*}
\end{proof}

The elements $s_\lambda[\Omega-1]$ are equal to $s_\lambda$ plus terms which
are of larger homogeneous degree.  We can order the partitions $\lambda <_* \mu$
if $|\lambda| < |\mu|$ or $|\lambda|=|\mu|$ and $\lambda <_{lex} \mu$
so that there is an inhomogeneous basis
$\{ \tilde{y}_\lambda \}$ indexed by partitions where
$\tilde{y}_\lambda = s_\lambda$ plus terms which are
of smaller homogeneous degree which has the property that
$\left< \tilde{y}_\lambda, s_\mu[\Omega-1] \right> = \delta_{\lambda\mu}$.
This is the graded dual basis to $\{ s_\lambda[\Omega-1] \}$
and the next proposition shows that the evaluations of $\tilde{y}_\lambda$
are rook characters.

\begin{prop} \label{prop:dual_basis}
Let $\{ \tilde{y}_\lambda \}$ represent the graded dual basis
to $\{s_\lambda[\Omega-1] \}$ with respect to the scalar product,
then $\phi_n( \tilde{y}_\lambda ) = h_{n - |\lambda|} s_\lambda$.
\end{prop}

\begin{proof}  Fix an integer $n$ and let $\mu \vdash n$.
Using Equations \eqref{eq:scharfthibon} and \eqref{eq:pieri_duals}, we compute
\begin{align*}
\left< \phi_{n}(\tilde{y}_\lambda), s_\mu \right>
&= \left< \tilde{y}_\lambda, s_\mu[\Omega] \right>\\
&= \left< \tilde{y}_\lambda, s_\mu[\Omega - 1 + 1] \right>\\
&= \sum_{\mu/\nu\in\mathcal{H}} \left< \tilde{y}_\lambda, s_\nu[\Omega - 1] \right>\\
\end{align*}
where the sum on the right hand side is over all partitions such
that $\mu/\nu$ is a horizontal strip.  That is, the sum on the right
hand side will be $1$ if $\mu/\lambda$ is a horizontal strip
and is equal to $0$ otherwise. Hence, the Schur expansion of
$\phi_{n}(\tilde{y}_\lambda)$ is equal to $h_{n-|\lambda|} s_\lambda$.
\end{proof}

\begin{corollary}
\label{corollary:PP_character_plethysm}
Fix an integer $k$ and let $\lambda$ and $\mu$ be partitions such that
$|\lambda|, |\mu| \leq k$, then
\[
\chi^\lambda_{\PP_k}(\mu) = \left< s_\lambda[\Omega-1], p_\mu \right>
\]
and, in particular, if $|\mu|<|\lambda|$ then $\chi^\lambda_{\PP_k}(\mu)=0$.
\end{corollary}
\begin{proof}
This follows from Theorem \ref{thm:power_sum_character}
and Theorem \ref{thm:dual_basis}, since
\begin{align*}
\left< s_\lambda[\Omega-1], p_\mu \right>
&= \left< s_\lambda[\Omega-1], \sum_{\nu : 1 \leq |\nu| \leq |\mu|}
\chi^\nu_{\PP_k}(\mu) \tilde{x}_\nu \right>\\
&= \chi^\lambda_{\PP_k}(\mu)~.\qedhere
\end{align*}
\end{proof}

We now give formulas for the entries in the matrices $L$ and $R$ in terms of plethysm.

\begin{prop}\label{prop:plethystic_formula_for_factorization}
	The matrices $L$ and $R$ for $\PP_k$ are given by
	\begin{align*}
		L_{\lambda,\mu}&=\left< s_\lambda[\Omega-1], s_\mu \right>\text{, and}\\
		R_{\lambda,\mu}&=\frac{1}{z_\lambda}\left< p_\lambda[\Omega-1], p_\mu \right>~.
	\end{align*}
\end{prop}

\begin{proof}
By linearity of the scalar product and the fact that the Schur basis is self-dual, we compute
\begin{align*}
	C_{\lambda,\mu}=\left< s_\lambda[\Omega-1], p_\mu \right>&= \sum_\alpha \left< s_\lambda[\Omega-1], s_\alpha \right> \left< s_\alpha, p_\mu \right>\\
	&=\sum_\alpha \left< s_\lambda[\Omega-1], s_\alpha \right> Y_{\alpha,\mu}.\\
	\intertext{By uniqueness of the matrix $L$, we have shown the desired formula for $L_{\lambda,\mu}$.
	Further using linearity of the plethysm $f[g]$ in $f$, we compute}
	C_{\lambda,\mu}=\left< s_\lambda[\Omega-1], p_\mu \right>&=\sum_\alpha \left<s_\lambda, p_\alpha \right> \frac{1}{z_\alpha}\left< p_\alpha[\Omega-1], p_\mu \right>\\
	&=\sum_\alpha Y_{\lambda,\alpha}\frac{1}{z_\alpha}\left< p_\alpha[\Omega-1], p_\mu \right>.
\end{align*}
By uniqueness of the matrix $R$, we have shown the desired formula for $R_{\lambda,\mu}$.
\end{proof}

In \cite[Theorem 5.20]{halverson_jacobson}, Halverson and Jacobson give a combinatorial method for
computing $R_{\lambda,\mu}$ as enumerating a certain set of symmetric diagrams in $\PP_k$.
Notice that a similar method for computing $L_{\lambda,\mu}$ would  give a combinatorial
interpretation of the Schur expansion of the plethysm $s_\lambda[\Omega-1]$. This would in turn
lead to a combinatorial interpretation for the decomposition of an irreducible $GL_n$ representation
restricted to the $n\times n$ permutation matrices (the ``restriction problem'').

Applying the duality between $\{\xt_\lambda\}$ and $\{s_\lambda[\Omega-1]\}$ from Theorem~\ref{thm:dual_basis},
the values $L_{\lambda,\mu}$ appear in the change of basis
from the $s_\lambda$ basis to the $\xt_\lambda$ basis.
The following proposition states this relationship precisely.

\begin{prop}
	For any partition $\mu$,
	\[s_\mu=\sum_\lambda L_{\lambda,\mu}\xt_\lambda.\]
\end{prop}

\subsection{Plethystic formulas for the structure coefficients} 
Reduced Kronecker coefficients are indexed by three partitions $\alpha$, $\beta$ and $\gamma$,
and are denoted by $\overline{g}_{\alpha\beta}^\gamma$. They arise as the stable values of
the Kronecker coefficients, $g_{(n-|\alpha|,\alpha), (n-|\beta|,\beta),(n-|\gamma|,\gamma)}$
as $n$ tends to infinity. They are an important object of study in algebraic combinatorics
and representation theory.

The last two authors \cite{OZ_character} showed that the reduced Kronecker coefficients
arise naturally as the structure constants of the inhomogeneous irreducible character
basis of the ring of symmetric functions, whose elements specialize to irreducible
characters of the symmetric groups. Pak and Panova \cite{PP2020} showed that reduced
Kronecker coefficients do not satisfy the saturation property. More recently,
Ikenmeyer and Panova \cite{IP2024} proved that every Kronecker coefficient can be
realized as a reduced Kronecker coefficient.

In the following two propositions, we give plethystic formulae for the
structure constants $\overline{g}_{\alpha\beta}^\gamma$ and $k_{\alpha\beta}^\gamma$ as they enlighten the relationship between them.

A formula for the reduced Kronecker coefficient is due to Littlewood
\cite[Theorem IX]{Littlewood}.  For partitions $\alpha, \beta, \gamma$,
\[
\overline{g}_{\alpha\beta}^\gamma = \sum_{\delta,\epsilon,\zeta, \sigma, \rho, \tau}
g_{\delta\epsilon\zeta} c^\alpha_{\delta\sigma\tau} c^\beta_{\epsilon\rho\tau} c^\gamma_{\zeta\rho\sigma}
\]
where sum in that expression is over all partitions, but implicitly
the Littlewood-Richardson coefficients
and Kronecker coefficients are equal to $0$ except when
$|\delta|=|\epsilon|=|\zeta|$, $|\delta|+|\sigma|+|\tau| = |\alpha|$,
$|\epsilon|+|\rho|+|\tau|=|\beta|$, and $|\zeta|+|\rho|+|\sigma|=|\gamma|$.

Translated into plethystic notation, Littlewood's formula takes the following form:

\begin{prop}
\label{prop:pleth_irred_structure}
For partitions $\alpha, \beta, \gamma$ such that $0 \leq |\gamma| \leqslant |\alpha|+|\beta|$,
\[
\overline{g}_{\alpha\beta}^\gamma = \left< s_\alpha[X] s_\beta[Y], s_\gamma[X+Y+XY] \Omega[XY] \right>
\]
\end{prop}

\begin{proof} We use the identities \eqref{eq:Schur_coproducts}
and \eqref{eq:Cauchy_properties}
where in our sums below each is implicitly over the partitions that are free.
\begin{align*}
&\left< s_\alpha[X] s_\beta[Y], s_\gamma[X+Y+XY] \Omega[XY] \right>\\
&=\sum c^\gamma_{\zeta\rho\sigma} \left< s_\alpha[X] s_\beta[Y], s_\sigma[X]s_\rho[Y]s_\zeta[XY] \Omega[XY] \right>\\
&= \sum c^\gamma_{\zeta\rho\sigma}
g_{\delta\epsilon\zeta} \left< s_\alpha[X] s_\beta[Y], s_\sigma[X]s_\rho[Y]s_\delta[X]s_\epsilon[Y] \Omega[XY] \right>\\
&= \sum c^\gamma_{\zeta\rho\sigma}
g_{\delta\epsilon\zeta} \left< s_\alpha[X] s_\beta[Y], s_\sigma[X]s_\rho[Y]s_\delta[X]s_\epsilon[Y] s_\tau[X] s_\tau[Y] \right>\\
&=\sum
g_{\delta\epsilon\zeta} c^\alpha_{\delta\sigma\tau} c^\beta_{\epsilon\rho\tau} c^\gamma_{\zeta\rho\sigma}
= \overline{g}_{\alpha\beta}^\gamma~. \qedhere
\end{align*}
\end{proof}


The following identity related to the Heisenberg (smash) product appears in
\cite[Theorem 1]{MS25} as the multiplicity of an irreducible
in the tensor of two irreducible representations
of the rook monoid.
\begin{equation} \label{eq:rook_kronecker}
k_{\alpha\beta}^\gamma = \sum_{\delta,\epsilon,\zeta, \sigma, \rho}
g_{\delta\epsilon\zeta} c^\alpha_{\delta\sigma} c^\beta_{\epsilon\rho} c^\gamma_{\zeta\rho\sigma}
\end{equation}
where (as was the case with Littlewood's formula for the reduced Kronecker coefficients)
the sum is over all partitions $\delta,\epsilon,\zeta, \sigma, \rho$
but the Littlewood-Richard coefficients and Kronecker coefficients vanish unless
$|\delta|=|\epsilon|=|\zeta|$, $|\alpha|=|\delta|+|\sigma|$, $|\beta|=|\epsilon|+|\rho|$
and $|\gamma|=|\zeta|+|\rho|+|\sigma|$.

We noticed that the right hand side also coincides with the same formula in \cite[Lemma 4.2]{Ying22}. Again, we can translate these coefficients into plethystic notation
because it more transparently shows the relationship between
the $\overline{g}_{\alpha\beta}^\gamma$ and $k_{\alpha\beta}^\gamma$ coefficients.

\begin{prop}\label{prop:pleth_rook_structure}
For partitions $\alpha, \beta, \gamma$ such that $\max(|\alpha|, |\beta|) \leq |\gamma| \leqslant |\alpha|+|\beta|$,
\[
k_{\alpha\beta}^\gamma = \left< s_\alpha[X] s_\beta[Y], s_\gamma[X+Y+XY] \right>
\]
\end{prop}

\begin{proof} Again, in the following calculation the sum is implicitly over the partitions that are free
and follows from applications of Equations \eqref{eq:Schur_coproducts}
and \eqref{eq:Cauchy_properties}.
\begin{align*}
&\left< s_\alpha[X] s_\beta[Y], s_\gamma[X+Y+XY] \right>\\
&=\sum c^\gamma_{\zeta\rho\sigma} \left< s_\alpha[X] s_\beta[Y], s_\sigma[X]s_\rho[Y]s_\zeta[XY] \right>\\
&= \sum c^\gamma_{\zeta\rho\sigma}
g_{\delta\epsilon\zeta} \left< s_\alpha[X] s_\beta[Y], s_\sigma[X]s_\rho[Y]s_\delta[X]s_\epsilon[Y] \right>\\
&=\sum
g_{\delta\epsilon\zeta} c^\alpha_{\delta\sigma} c^\beta_{\epsilon\rho} c^\gamma_{\zeta\rho\sigma}
= k_{\alpha\beta}^\gamma~.\qedhere
\end{align*}
\end{proof}

We note that Equation \eqref{eq:rook_kronecker} implies that
if $|\alpha|=|\beta|=|\gamma|$, then $k_{\alpha\beta}^\gamma=g_{\alpha\beta\gamma}$.
It also implies that if $|\alpha|+|\beta|=|\gamma|$, then
$k_{\alpha\beta}^\gamma=c_{\alpha\beta}^\gamma$.  In that sense,
we say that the structure coefficients for the characters of the
rook monoid are an interpolation between the Kronecker coefficients
(the smallest degree component)
and the Littlewood-Richardson coefficients (the highest degree component).

Any algebra morphism applied to the $\xt$-basis will
have the same structure coefficients.  To characterize the basis,
we specify that a subset of the elements are a set of generators of the algebra.
This is stated formally in the following theorem.

\begin{theorem}\label{theorem:structure_characterization}
The set $\{ \xt_\lambda \}$ is the unique basis satisfying Proposition \ref{prop:structure}
such that $\xt_{1^r} = e_r$ for all $r \geq 1$.
\end{theorem}

\begin{proof}  We note that since $\xt_{1^r} = e_r$,
then $s_\lambda = \det | \xt_{1^{\lambda_i' -i+ j}} |= \det | e_{{\lambda_i'-i + j}} |$
by the Jacobi-Trudi formula for Schur functions \cite[Equation (3.5)]{Macdonald}.
Since for the highest degree components, the smash product coincides with the usual
product of Schur functions,
we have that $s_\lambda = \xt_\lambda$ plus terms involving $\xt_\mu$ indexed by partitions
$\mu$ with $|\mu|<|\lambda|$.  We conclude that $\xt_\lambda$ is equal to
$s_\lambda$ minus these same terms. By induction on the size of the
partition $\lambda$, $\xt_\lambda$ is determined by these structure coefficients.
\end{proof}

\section{A multiset generalization of the Kostka numbers}
\label{section:multiset}
The Kostka numbers arise when we write the complete homogeneous symmetric function
$h_\mu$ in terms of the Schur functions, $s_\lambda$,
\[
h_\mu=\sum_{\lambda \vdash |\lambda|}K_{\lambda\mu}s_\lambda.
\]
In this section, we show that when $h_\mu$ is expanded in the
$\{\xt_\lambda\}$ basis, the resulting coefficients are given by a multiset
generalization of the Kostka numbers.

Because of the characterization from
Theorem \ref{thm:dual_basis} and Proposition \ref{prop:dual_basis}
of $\{ \xt_\lambda \}$ as the graded dual basis to the elements
$\{ s_\lambda[\Omega-1] \}$, we have the two expansions:
\[
h_\mu = \sum_{\lambda : |\lambda| \leq |\mu|} \left< h_\mu, s_\lambda[\Omega-1]\right> \xt_\lambda
\qquad\hbox{and}\qquad
s_\lambda[\Omega-1] = \sum_{\mu : |\lambda| \leq |\mu|} \left< h_\mu, s_\lambda[\Omega-1]\right> m_\mu~.
\]
Moreover, asserting that the coefficient of $\xt_\lambda$ in $h_\mu$ is equal
to the coefficient of $m_\mu$ in $s_\lambda[\Omega-1]$ is equivalent to
the statement that $\{ \xt_\lambda \}$ is the graded dual basis
to $\{ s_\lambda[\Omega-1] \}$.

A {\it multiset} $S$ is a collection of elements where repetition is allowed.
A multiset partition $\pi = \lcr S_1,S_2,\ldots,S_\ell\rcr$ of a multiset $S$
is a multiset whose elements $S_i$ are each multisets for $1 \leq i \leq \ell$
and such that $S = S_1 \uplus S_2 \uplus \cdots \uplus S_\ell$.
We will put a canonical (albeit arbitrary) total order on the the multisets
so that $S_i \leqslant S_j$ if the size of $S_i$ is smaller than the size of $S_j$
and, if they both have the same size, then the reading word of the elements of $S_i$ is a word
that is lexicographically smaller or equal to the reading of the elements in $S_j$.

A {\it column strict tableau of shape $\lambda$}
is a map  $T$ from the set ${\mathrm{cells}}(\lambda)$ to a set of labels
such that
\begin{enumerate}
\item $T(i,j) \leqslant T(i+1,j)$ if both $(i,j)$ and $(i+1,j)$ are
in ${\mathrm{cells}}(\lambda)$ and
\item
$T(i,j) < T(i,j+1)$ if both $(i,j)$ and $(i,j+1)$ are
in ${\mathrm{cells}}(\lambda)$.
\end{enumerate}
A {\it multiset tableau of a multiset $S$} is a column strict tableau $T$
such that $\lcr T(c) : c \in {\mathrm{cells}}(\lambda) \rcr$
is a multiset partition of $S$.

The number of multiset tableaux of the multiset
$\lcr 1^{\mu_1}, 2^{\mu_2}, \ldots, \ell^{\mu_\ell} \rcr$
such that $T$ is shape $\lambda$ will be denoted
${\mathcal K}_{\lambda\mu}$.  We note that if $|\lambda| = |\mu|$
then ${\mathcal K}_{\lambda\mu}$ is equal to the Kostka
coefficient $K_{\lambda\mu}$.  We also note that
${\mathcal K}_{\lambda\mu} = 0$ unless $|\lambda|\leq|\mu|$.

\begin{theorem}\label{thm:combinatorial}
For each partition $\mu$,
\[
{\mathcal K}_{\lambda\mu} = \left< h_\mu, s_\lambda[\Omega - 1] \right>~.
\]
\end{theorem}
\begin{proof}
We note that since $\{ h_\lambda \}$ and $\{m_\lambda\}$
are dual bases, then for any symmetric function $f$, the expression
$\left< f, h_\mu \right>$ is equal to the coefficient
of the monomial symmetric function $m_\mu$ in $f$.  This
expression is in turn equal to the coefficient
of ${\mathbf x}^\mu:=x_1^{\mu_1} x_2^{\mu_2} \cdots x_\ell^{\mu_\ell}$
in the expansion $f[X]$.  Therefore we have that $\left< h_\mu, s_\lambda[\Omega-1]\right>$
is equal to the coefficient of ${\mathbf x}^\mu$ in
the series $s_\lambda[\Omega[X]-1]$.  The expression $s_\lambda[\Omega[X]-1]$ is a series, but for
the purposes of this proof we need only need compute with the monomials of degree
smaller than or equal to $|\mu|$ since others will not contribute to the coefficient
of ${\mathbf x}^\mu$.

Expanding Equation \eqref{eq:cauchy} as a geometric series we have that
\begin{equation}\label{eq:cauchy_series}
\Omega[X]-1 = \prod_{i\geq 1} \frac{1}{1-x_i} -1 = \sum_{S} {\mathbf x}^S
\end{equation}
where the sum is over all non-empty multisets $S$.  Here
we have used the notation ${\mathbf x}^S$ to represent the monomial
$\prod_{i \in S} x_i$.
Implicitly, we will order our monomials with the same
order that we used for the multisets.
The definition of $s_\lambda[\Omega[X]-1]$ is to replace
each instance of $p_r$ in $s_\lambda$ with
$\sum_{S} ({\mathbf x}^{S})^r$ and this is equal to the monomial
expression for $s_\lambda[Y]$ where each variable $y_i$
is replaced with a monomial ${\mathbf x}^S$.

The monomial expansion of the Schur function
\cite[Equation (5.12)]{Macdonald} is
\[
s_\lambda[Y] = \sum_T {\mathbf{wt}}(T)
\]
where the sum is over all column strict tableaux $T$ of shape
$\lambda$ and whose entries are labels from the index set of
the alphabet $Y$.  Using Equation \eqref{eq:cauchy_series}, when $Y = \Omega[X]-1$
we allow each $T(c)$ to be a non-empty multiset.
The expression ${\mathbf{wt}}(T)$ is a
monomial equal to $\prod_{c \in {\mathrm{cells}}(\lambda)} y_{T(c)} =
\prod_{c \in {\mathrm{cells}}(\lambda)} {\mathbf x}^{T(c)}$.
The number of tableaux contributing a weight equal to
${\mathbf x}^\mu$ is precisely the number of
multiset tableaux of shape $\lambda$ such that
$\lcr T(c) : c \in {\mathrm{cells}}(c) \rcr$ is a
multiset partition of the multiset
$\lcr 1^{\mu_1}, 2^{\mu_2}, \ldots, \ell^{\mu_\ell} \rcr$.
\end{proof}

\begin{example} As an example, we compute the expansion of a the
element $h_{32}$ in the $\xt$-basis as :
\begin{align*}
h_{32} = \xt_{32} &+ \xt_{41} + \xt_{5} + 2 \xt_{211} + 2 \xt_{22} + 5 \xt_{31} + 3 \xt_{4}\\
&+ 3 \xt_{111} + 9 \xt_{21}  + 6 \xt_{3} + 5 \xt_{2} + 5 \xt_{11} + \xt_{1}~.
\end{align*}

In particular the coefficient of $\xt_{21}$ in this expression is $9$
and is equal to the number of the following set of multiset tableaux
(expressed using the `French' convention with the largest row on the bottom of the diagram).
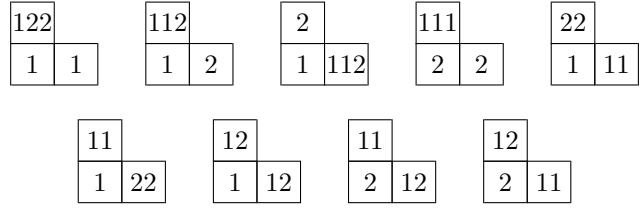
\[
\young{{\Tiny 122}\cr1&1\cr}\hskip .25in
\young{{\Tiny 112}\cr1&2\cr}\hskip .25in
\young{2\cr1&{\Tiny 112}\cr}\hskip .25in
\young{{\Tiny 111}\cr2&2\cr}\hskip .25in
\young{22\cr1&11\cr}
\]
\[
\young{11\cr1&22\cr}\hskip .25in
\young{12\cr1&12\cr}\hskip .25in
\young{11\cr2&12\cr}\hskip .25in
\young{12\cr2&11\cr}
\]

\end{example}

\section{Using Sage to compute with the rook character basis}

\lstset{
    basicstyle=\footnotesize\ttfamily, 
    breaklines=true                    
}

In this section we provide a brief tutorial of computations using the rook character
basis using Sage.  The definition implemented there is through the change of basis with
the power sums as given in Equation \eqref{eq:xt_to_power}.

The $\xt_\lambda$ is defined as one of the optional bases in the ring of symmetric functions.
It can be added to the namespace by the command:
{\Small\begin{verbatim}
sage: SymmetricFunctions(QQ).inject_shorthands('all', verbose=False)
\end{verbatim}}
or by defining it directly:
{\Small\begin{verbatim}
sage: xt = SymmetricFunctions(QQ).rook_character_basis()
\end{verbatim}}
or by the shorthand:
{\Small\begin{verbatim}
sage: xt = SymmetricFunctions(QQ).xt()
\end{verbatim}}

\vskip .2in
The rook character basis is defined by its
expansion in the power sum basis (Example \ref{ex:xt_to_power}).
{\Small\begin{verbatim}
sage: p(xt[2, 1])
p[1] - p[1, 1] + 1/3*p[1, 1, 1] - 1/3*p[3]
\end{verbatim}}

\vskip .2in
The character table for the rook monoid $R_3$ can be computed by
evaluating the basis elements at the eigenvalues of
permutation matrices.
{\Small\begin{verbatim}
sage: [[xt(lam).eval_at_permutation_roots(mu)\
....:  for d1 in range(4) for mu in Partitions(d1)]\
....:  for d2 in range(4) for lam in Partitions(d2)]
[[1, 1, 1, 1, 1, 1, 1],
 [0, 1, 0, 2, 0, 1, 3],
 [0, 0, 1, 1, 0, 1, 3],
 [0, 0, -1, 1, 0, -1, 3],
 [0, 0, 0, 0, 1, 1, 1],
 [0, 0, 0, 0, -1, 0, 2],
 [0, 0, 0, 0, 1, -1, 1]]
\end{verbatim}}

\vskip .2in
The character table for the propagating partition algebra
is the matrix of change of basis coefficients from the power sum basis
to the rook character basis.
The character table for $P\!P_3$ in Example \ref{ex:PPk_char_table} can be computed
using Theorem \ref{thm:power_sum_character} using the following command.
{\Small\begin{verbatim}
sage: [[xt(p(mu)).coefficient(lam)\
....:  for d1 in range(4) for mu in Partitions(d1)]\
....:  for d2 in range(4) for lam in Partitions(d2)]
[[1, 0, 0, 0, 0, 0, 0],
 [0, 1, 1, 1, 1, 1, 1],
 [0, 0, 1, 1, 0, 1, 3],
 [0, 0, -1, 1, 0, 1, 3],
 [0, 0, 0, 0, 1, 1, 1],
 [0, 0, 0, 0, -1, 0, 2],
 [0, 0, 0, 0, 1, -1, 1]]
\end{verbatim}}

Other character bases such as the
irreducible character basis, $\st_\lambda$, and the induced trivial character
basis (or permutation characters), ${\tilde h}_\mu$, are also implemented.
The coefficient of $\xt_\lambda$ in ${\tilde h}_\mu$ is the Kostka coefficient $K_{\lambda\mu}$.
The coefficient of $\st_\lambda$ in $\xt_\mu$ is $1$ if and only if $\mu/\lambda$ is a horizontal strip.
For these results see \cite{AS_specht, OZ_hopf}.

{\Small\begin{verbatim}
sage: xt(ht[3, 2])
xt[3, 2] + xt[4, 1] + xt[5]
sage: xt(ht[4, 2, 2, 1]).coefficient([5, 4])
3
sage: SemistandardTableaux(shape=[5,4], eval=[4,2,2,1]).cardinality()
3
sage: st(xt[3,2])
st[2] + st[2, 1] + st[2, 2] + st[3] + st[3, 1] + st[3, 2]
sage: s[3,2](s[1] + 1)
s[2] + s[2, 1] + s[2, 2] + s[3] + s[3, 1] + s[3, 2]
\end{verbatim}}

\vskip .2in
The single cell Pieri rule is given in \cite[Theorem 2.16]{MS21}.
{\Small\begin{verbatim}
sage: xt[3,2]*xt[1]
xt[2, 2, 1] + xt[3, 1, 1] + 2*xt[3, 2] + xt[3, 2, 1] + xt[3, 3] + xt[4, 1] + xt[4, 2]
\end{verbatim}}

More generally the structure coefficients are stated
in \cite[Theorem 1]{MS25} and agree with the smash (co)product of \cite{AFM, Ying20, Ying22}
on the Schur basis.

{\Small\begin{verbatim}
sage: lam = Partition([3,2]); mu = Partition([2,1,1])
sage: (xt(lam)*xt(mu)).coefficient([3, 1, 1, 1, 1])
8
sage: X = tensor([s[1], s.one()]); Y = tensor([s.one(), s[1]])
sage: s[3,1,1,1,1](X+Y+X*Y).monomial_coefficients()[(lam,mu)]
8
\end{verbatim}}

\vskip .2in
The ring of symmetric functions is a Hopf algebra.  The coproduct on the rook character basis
is the same as that on the Schur basis and the coefficients follow the Littlewood-Richardson rule
\cite[Corollary 14]{OZ_hopf}.
{\Small\begin{verbatim}
sage: xt[2,2].coproduct()
xt[] # xt[2, 2] + xt[1] # xt[2, 1] + xt[1, 1] # xt[1, 1] + xt[2] # xt[2]
 + xt[2, 1] # xt[1] + xt[2, 2] # xt[]
sage: s[2,2].coproduct()
s[] # s[2, 2] + s[1] # s[2, 1] + s[1, 1] # s[1, 1] + s[2] # s[2]
 + s[2, 1] # s[1] + s[2, 2] # s[]
\end{verbatim}}

\bibliographystyle{plain}
\bibliography{bibliography}

\end{document}